\documentclass[11pt,amstex]{article}

\usepackage{color,comment}              
\usepackage{graphicx,setspace}
\usepackage[]{amsmath}
\usepackage{amssymb,mathtools}
\usepackage{fullpage}
\usepackage{hyperref}
\usepackage{enumerate, amsfonts, amsthm, bbm}
\usepackage{caption}

\definecolor{MyDarkBlue}{rgb}{0,0.08,0.50}
\definecolor{BrickRed}{rgb}{0.65,0.08,0}

\hypersetup{
colorlinks=true,       
    linkcolor=MyDarkBlue,          
    citecolor=BrickRed,        
    filecolor=red,      
    urlcolor=cyan           
}
\numberwithin{equation}{section}

\usepackage{amssymb,graphicx,tabularx}

\newcommand{\ba}{{\mathbf a}}
\newcommand{\bA}{{\mathbf A}}
\newcommand{\bb}{{\mathbf b}}
\newcommand{\bc}{{\mathbf c}}
\newcommand{\be}{{\mathbf e}}
\newcommand{\bn}{{\mathbf n}}

\newcommand{\bq}{{\mathbf q}}

\newcommand{\bI}{{\mathbf I}}
\newcommand{\bV}{{\mathbf V}}
\newcommand{\bW}{{\mathbf W}}
\newcommand{\bx}{{\mathbf x}}
\newcommand{\bone}{{\bf 1}}

\newcommand{\blambda}{\boldsymbol\lambda}
\newcommand{\bmu}{\boldsymbol\mu}
\newcommand{\bpi}{\boldsymbol\pi}
\newcommand{\btau}{\boldsymbol\tau}
\newcommand{\smult}{\circ}

\def\today{\number\day\space\ifcase\month\or January\or February\or
March\or April\or May\or June\or July\or August\or September\or
October\or November\or December\fi\space\number\year}

\newcommand{\cov}{\mathop{\rm cov}\nolimits}
\newcommand{\var}{\mathop{\rm var}\nolimits}

\newcommand{\bbbJ}{\mathbb{J}}
\newcommand{\bbbE}{\mathbb{E}}
\newcommand{\bbbP}{\mathbb{P}}
\newcommand{\bbbR}{\mathbb{R}}

\newcommand{\calT}{{\cal T}}

\newcommand\lstar{{\overline{\lambda}}}

\def\Frac#1#2{\frac
{
 {\raise.6ex
 \hbox{$\disp#1$}}
}
{
 {\lower.6ex
 \hbox{$\disp#2$}}
 }
}

 \newtheorem{theorem}{Theorem}
 
 \newtheorem{lemma}{Lemma}

 \newtheorem{corollary}{Corollary}

 \theoremstyle{definition}
 \newtheorem{example}{Example}

\newcommand\Rpos{\bbbR_+}

\def\calP{{\cal P}}
\def\d{{\rm d}}
\def\e{{\rm e}}
\def\E{{\bbbE}}

\newcommand{\eqan}[1]{\begin{align} #1 \end{align}}

\newcommand{\RRR}{\mathcal{D}}

\begin{document}
\newcommand{\Bin}{\mathop{\rm{Bin}}\nolimits}
\newcommand{\card}{\mathop{\rm{card}}\nolimits}
\newcommand{\Exp}{\mathop{\rm{Exp}}\nolimits}
\newcommand{\rur}{\rule{\breite}{0.06mm}\vspace{1mm}\\}
\newcommand{\rP}{{\rm P}}
\newcommand{\Pik}{P_k}
\newcommand{\Js}{J}
\newcommand{\Nbp}{N_\Js}


\newcommand{\wtt}{\widetilde{t}}
\newcommand{\btT}{\widetilde{\mathbf T}}
\newcommand{\btI}{\widetilde{\bI}}
\newcommand{\btV}{\widetilde{\bV}}
\newcommand{\btW}{\widetilde{\bW}}
\newcommand{\btmu}{\widetilde{\bmu}}
\newcommand{\bttau}{\widetilde{\btau}}
\newcommand{\btq}{\widetilde{\bq}}
\newcommand{\LS}{{\calP}}
\newcommand{\talpha}{\widetilde{\alpha}}
\newcommand{\tN}{\widetilde{N}}
\newcommand{\tR}{\widetilde{R}}
\newcommand{\tv}{\widetilde{v}}
\newcommand{\tw}{\widetilde{w}}
\newcommand{\tu}{\widetilde{u}}
\def\by{{\bf y}}
\newcommand{\eqd}{=^d} 

\title{Second Order Properties of Thinned Counts\\
in Finite Birth--Death Processes}   
\author{
Daryl.\ J.\ Daley\footnote{Department of Mathematics and Statistics, The University of Melbourne. 
}, \
Yoni Nazarathy\footnote{School of Mathematics and Physics, The University of Queensland. \tt y.nazarathy@uq.edu.au.}, 
Jiesen Wang\footnote{Korteweg-de Vries Institute for Mathematics, University of Amsterdam, Amsterdam, The Netherlands. \tt  j.wang2(at)uva.nl.}
}

\maketitle

\abstract{
%
The paper studies the counting process arising as a subset of births and deaths in a birth--death process on a finite state space.  Whenever a birth or death occurs, the process is incremented or not depending on the outcome of an independent Bernoulli experiment whose probability is a {\color{black}function of the state of the birth and death process and also depends on whether it is a birth or death that has occurred}.
We establish a formula for the asymptotic variance rate of this process, also presented as the ratio of the asymptotic variance and the asymptotic mean. Several examples including queueing models illustrate the scope of applicability of the results. An analogous formula for the countably infinite state space is conjectured and tested.
}

\section{Introduction}
\label{sec:intro}

We first became aware of the second-order properties discussed in this note in various queueing models in which the simpler cases can be cast purely in a birth--death process setting where computation is more readily effected. In such cases there is a phenomenon, called BRAVO {\color{black} (Balancing Reduces Asymptotic Variance of Outputs)}, that appears in the limit as the state space grows. {\color{black} For simple output counting processes as in the case of M/M/1/K queues, the study of BRAVO simply requires considering the counting process of deaths in a birth-death process. Yet in more complex models, a more general counting process is required.} 
In this paper we exhibit more general processes $N_q$ defined by counting the births and deaths that remain after independent thinnings of these two types of events in a birth--death process and for which the BRAVO characteristics may or may not survive the thinning procedure.

Let $\{Q(t):t\ge 0\}$ be an irreducible 
continuous-time birth--death process on the state space 
${\mathbb J} := \{0,1,\ldots,\Js\}$ with $J < \infty$. 
Any realization of $Q$ is expressible in the form
\begin{equation} \label{eq:Qdecomp}
Q(t) = Q(0) ~+~ N_+(0,t] ~-~ N_-(0,t],
\end{equation}
where the counting measures $N_+$ and $N_-$ have unit atoms at the
instants of births and deaths respectively.  Thus, $N_+$ is the process of births in $Q$, and $N_-$ the process of deaths in $Q$.

We thin the sum-process $N_+ + N_-$ to obtain $N_q$ as follows.  Let
$\{q^+_j:j\in\mathbb{\Js}\}$ and
$\{q^-_j:j\in\mathbb{\Js}\}$ be two
families of $[0,1]$-valued constants, where there is at least one $j \in {\mathbb J}$ where either $q_j^+ > 0$, or $q_j^- > 0$, or both. Note that at the endpoints we set $q_0^- = 0$ and 
$q_J^+ = 0$. 

For each atom $s$ of $N_+$ let $\tilde J^+(s) = 1$ with probability $q^+_{Q(s^-)}$, or let $\tilde J^+(s) = 0$ otherwise. Similarly, for each atom $s$ of $N_-$ let $\tilde J^-(s) = 1$ with probability $q^-_{Q(s^-)}\,$, or let  $\tilde J^-(s) =0$ otherwise. Hence the probability of having atoms for $\tilde J^+$ and $\tilde J^-$ depend on the state of $Q$ just before the jump. Now let
\begin{equation}
\label{eq:19StanleySt}
N_q(0,t] = \int_0^t \tilde J^+(u)\,N_+(\d u)
~+~ 
 \int_0^t \tilde J^-(u)\,N_-(\d u),
\end{equation}
where conditional on $\{Q(u): u< t'\}$ for any finite $t'$, all indicator r.v.s $\tilde J^\pm(t)$ for $t < t'$ are mutually independent. 

This paper studies the asymptotic variance also represented via the asymptotic index of dispersion of $N_q$,
 \begin{equation} 
 \label{eq:asymvar}
\RRR := \lim_{t\to\infty} \frac{\var N_q(0,t]}{\E\big[N_q(0,t]\big]},
\end{equation}
{\color{black} namely the limit of the variance of $N_q(0,t]$ divided by the expectation of $N_q(0,t]$.} 
In mathematical biology, for a general stationary counting process $N_q$,
$\RRR$ is known as the Fano factor, as for example in Eden and Kramer
\cite{EdenK}.  
The denominator and numerator here are asymptotically 
linear in $t\ge 0$, and 
the limit in \eqref{eq:asymvar} is always finite and positive.
Our main result, Theorem~\ref{thm:main}, presents a formula for
$\RRR$ in terms of $\{q^\pm_j:j\in\mathbb{\Js}\}$ and the birth- and
death-rates of $Q$. To our knowledge this condensed explicit formula is new. Yet we mention that since the state space is finite, computable matrix based expressions for $\RRR$ can be obtained by considering the process as a MAP (Markovian Arrival Process); see for example~\cite{he2014fundamentals}. 

In the special case of a pure death-counting process (all $q_i^+=0$), no thinning
($q^-_i \equiv 1$), a similar formula was derived in \cite{NazarathyWeiss0336} where the BRAVO effect was first noted. {\color{black} BRAVO is a phenomenon observed in certain queueing systems, where if the arrival rate equals (exactly or approximately) the service capacity, then the asymptotic variance of the departure (outputs) process counts reduces in comparison to cases where the arrival rate and service capacity are different. For simple examples of queues, such as the M/M/1/K queue, the study of BRAVO requires study of the process $N_-$, as this is the output process of the queue. However, when considering more complex examples, such as queues with reneging (see Example~1 in Section~\ref{sec:exx}), to study the output process we need to introduce state dependent thinning as appearing in the second term of \eqref{eq:19StanleySt}.} 

In a sequence of papers associated with BRAVO \cite{daley2011revisiting, daley2013bravo, glynn2023heavy, hautphenne2013second, nadarajah2018precise, nazarathy2011variance,nazarathy2022busy, NazarathyWeiss0336}, it was shown that for non-small $J$,  it holds that $\RRR \approx 1$, except for `balanced` singular cases in which the arrival rate equals the service capacity. In those cases $\RRR \approx {\cal B}$ where for many standard queueing models, the BRAVO constant ${\cal B} < 1$. This constant equals $2/3$ for simple single-server Markovian queues with a finite buffer, \cite{hautphenne2013second, NazarathyWeiss0336}; equals $2(1-2/\pi)\approx 0.7268$ for infinite buffer cases with infinite buffers,
\cite{al2011asymptotic, nazarathy2022busy}; and equals $\approx 0.6093$ in many-server queues under
the so-called Halfin--Whitt scaling regime, \cite{daley2013bravo}. Note also a correction paper for some of the computations associated with that result, \cite{nadarajah2018precise}.

Further, for queues with renewal input (i.e.\ GI/$\cdot/\cdot)$ systems), the first multiple, $2$, in $2(1-2/\pi)$ is replaced by the sum of the squared coefficients of variation of the inter-arrival and service times driving a GI/G/$s$ queue, \cite{al2011asymptotic}; similarly in the finite case, as conjectured and numerically tested, \cite{nazarathy2011variance}, ${\cal B}$ equals the sum of the squared coefficients of variations divided by $3$.
All of these examples motivate the study of the asymptotic index of dispersion, $\RRR$.

For simple queueing models with reneging, the question of the existence of BRAVO remains.  When such queueing processes are represented by birth--death models, a thinning mechanism is required, since some deaths count as departures but others are reneging customers.  This motivated us to seek a general formula for
$\RRR$. Here we prove our formula for $\RRR$
for the finite ${\mathbb J}$ case. We also explore its use for the countably infinite case, and leave the exact conditions of when this infinite case holds for further research.  Our derivation follows from detailed regenerative arguments, including the manipulation of moment expressions via matrix algebra. 
The derivation also hinges on an explicit inverse of a matrix, which to the best of our knowledge has not appeared earlier in this form, and can in principle be used for other birth and death related results.

The paper is organized as follows. In Section~\ref{sec:main-result} we present the main result, Theorem~\ref{thm:main}. We then prove the main result in Section~\ref{sec:finite-proof}. The proof hinges on an explicit formula for the inverse of a matrix which may be of independent interest and is also presented as a lemma in Section~\ref{sec:finite-proof}. 
We close with examples in Section~\ref{sec:exx}.


\section{Setup and Main Result}
\label{sec:main-result}

Let $\{Q(t), t \ge 0\}$ be an irreducible continuous-time birth--death (BD)
Markov chain on finite state space $\bbbJ := \{0,1,\ldots,\Js\}$, for finite positive integer $J$.
The birth rates are $\lambda_0,\ldots, \lambda_J$ and death rates are $\mu_0, \ldots ,\mu_J$ with
$\lambda_J = \mu_0=0$ and all other rates positive. 


The stationary (and limiting) probabilities $\pi_i := \lim_{t \to \infty} \bbbP\{Q(t) = i\}$
satisfy the partial balance relations
\begin{equation} 
\label{detbalance}
\pi_i \mu_i = \pi_{i-1} \lambda_{i-1} \qquad \text{for} \qquad i \in \bbbJ \setminus \{0\}, 
\end{equation} 
and therefore
\begin{equation} \label{eq:21}
 \pi_i  
  = \frac{  \prod_{k=1}^i (\lambda_{k-1}/\mu_k) }
 {\sum_{j \in \bbbJ}  \prod_{k=1}^j (\lambda_{k-1}/\mu_k) }
 \,=\, \pi_0 \prod_{k=1}^i \frac{\lambda_{k-1}}{\mu_k}
\qquad
\text{for}
\qquad
i\in\bbbJ,
\end{equation}
with empty products being unity (so $\pi_0 \ne 0$). We denote the
cumulative distribution elements as 
\begin{equation} \label{Pkdn}
\Pik := \sum_{i=0}^k \pi_i,
\qquad
\text{for}
\qquad
k \in \bbbJ.
\end{equation}
Denote the {\em raw rate} of {\color{black} births and deaths} as,
\begin{equation}
\label{lambdaQ}
\lstar_\star =
\sum_{i \in \bbbJ} \pi_i (\lambda_i + \mu_i) = 
2\sum_{i\in \bbbJ} \pi_i\lambda_i,
\end{equation}
where the second equality follows from \eqref{detbalance}. The rate $\lstar_\star $ satisfies, 
\[
\lim_{t \to \infty} \frac{\E\big[N_+(0,t] + N_-(0,t]\big]}{t} = \lstar_\star.
\]
With each state $i \in \bbbJ$, associate the probabilities
$q^\pm_i \in [0,1]$ and let $N_q(\cdot)$ denote the thinned sum-process based on $Q$ as described
in \eqref{eq:19StanleySt}. With this, define the {\em thinned rate} as
\begin{equation}
\label{eq:lambda-bar-def-first}
\lstar
= \sum_{i \in \bbbJ} \pi_i (\lambda_i q^+_i  +
\mu_i q^-_i).
\end{equation}
This rate $\lstar$ satisfies,
\[
\lim_{t \to \infty} \frac{\E\big[N_q(0,t] \big]}{t} = \lstar.
\]
%
Note also that ergodicity shows that in the long run $N_q(\cdot)$ counts a proportion $\varpi$
of all births and deaths, where 
\begin{equation} 
\label{countratio}
\varpi = \frac{\lstar}{\lstar_\star} = \lim_{t \to \infty} \frac{\E\big[N_q(0,t] \big]}{\E\big[N_+(0,t] + N_-(0,t]\big]} = 
\lim_{t \to \infty} \E \Big[\frac{N_q(0,t] }{N_+(0,t] + N_-(0,t]}\Big]. 
\end{equation}
Also note that, 
\begin{equation} 
\label{eq:EN22}  
\lstar = 
\sum_{i=0}^{J-1} \pi_i\lambda_i q^+_i  +
\sum_{i=1}^J \pi_i\mu_i q^-_i 
= \sum_{i=0}^{J-1} \pi_i \lambda_i(q^+_i + q^-_{i+1}),
\end{equation}
where the last equality follows using \eqref{detbalance}. 
%

We also use the partial sums
\begin{equation}
\lstar_k :=
  \sum_{i=0}^k \pi_i\lambda_i q^+_i+
\sum_{i=1}^k \pi_i\mu_i q^-_i
= \pi_k \lambda_k q_k^+ + \sum_{i=0}^{k-1} \pi_i \lambda_i(q^+_i + q^-_{i+1}),
\label{eq:little-lambda-k-def}
\end{equation}
and observe that 
$\lstar_J = \lstar$. 
Further, when normalized as
\begin{equation} \label{eq:Lambda}
 \Lambda_k := \frac{\lstar_k}{\lstar}
  \,=\,
\frac{ \pi_k\lambda_kq^+_k
+ \sum_{i=0}^{k-1} \pi_i \lambda_i (q^+_i+q_{i+1}^-)} {\lstar}
  \qquad
  \text{for}
  \qquad k\in\bbbJ,
\end{equation}
$\{\Lambda_k\}$ defines a cumulative distribution in the same manner that $\{P_k\}$ does.  



Theorem \ref{thm:main} presents a formula for the asymptotic index of dispersion 
$\RRR$ at \eqref{eq:asymvar} in terms of the distribution functions $\{P_k\}$
and $\{\Lambda_k\}$, the counting rate $\lstar$ and the sequences
$\{q_k^\pm\}$, $\{\lambda_k\}$ and $\{\pi_k\}$.
  If $N_q(\cdot)$ were a Poisson process we should have $\RRR=1$.
This is certainly not the case in general. 
Our formula 
directly generalizes Theorem 1 of \cite{NazarathyWeiss0336} which handles the pure death-counting process case of $q_i^- \equiv 1$ and
$q_i^+ \equiv 0$; so $N_q(0,t] = N_-(0,t]$ as in \eqref{eq:Qdecomp}. The result in \cite{NazarathyWeiss0336} is  proved by combining results for Markovian Arrival Processes with limit results
for Markov Modulated Poisson Processes; the method there requires the
retention probabilities $q_i^\pm$ to be $\{0,1\}$-valued.  Our proof is
quite different: it is more elementary, albeit with tedious calculations;
importantly, it no longer requires $q_i^- = 1$ and $q_i^+ = 0$.




\begin{theorem} 
\label{thm:main} 
Consider the finite state space case where irreducible $Q(\cdot)$ follows any initial distribution on $\bbbJ$. The asymptotic index of dispersion of $N_q$ is,
\begin{equation}    
\label{eq:RJgenPformula}
\RRR = 1+2 \sum_{k=0}^{J-1} R_k,
\end{equation} 
where,
\begin{equation}
\label{eq:R-k}
R_k = (P_k -\Lambda_k)\bigg(
   \frac{\lstar (P_k-\Lambda_k) }{\pi_{k}\lambda_k} + q^+_k - q^-_{k+1} \bigg).
\end{equation}
\end{theorem}

\subsection*{The Infinite State Space Case}

While Theorem~\ref{thm:main} is limited to a finite state space, a natural extension is to consider a countably infinite state space where $\lambda_i > 0$ for all all $i \ge 0$, $\mu_i>0$ for all $i \ge 1$, and $q^\pm_i$ are defined for all $i \ge 0$ with $q_0^- = 0$. In such a case, we also require the standard stability (positive {\color{black} recurrence}) condition,
\begin{equation}
\label{eq:stability}
\sum_{j \in \bbbJ} \prod_{k=1}^j \frac{\lambda_{k-1}}{\mu_k}< \infty,
\end{equation}
{\color{black} see for example Section~1.3 in \cite{weiss2021scheduling}. } 
In this infinite state space case, $P_k$, $\lambda_k$, $\Lambda_k$, and $\lstar$ extend naturally as above with,
\begin{equation} 
\label{eq:EN22-inf} 
\lstar = \sum_{i=0}^{\infty} \pi_i \lambda_i(q^+_i + q^-_{i+1}).
\end{equation}
Now using $R_k$ as in \eqref{eq:R-k} we define the formal expression,
\begin{equation}
\label{eq:Dinfbd}
\RRR_\infty = 1+2 \sum_{k=0}^{\infty} R_k.
\end{equation}
While in this paper we do not prove that $\RRR_\infty$ equals the asymptotic index of dispersion $\RRR$ of \eqref{eq:asymvar}, we conjecture that under suitable regularity conditions, the expression for $\RRR_\infty$ evaluates to the asymptotic index of dispersion in the countably infinite state space case. Some insightful examples and numerical results involving the expression \eqref{eq:Dinfbd} are presented in Section~\ref{sec:exx}. 

\subsection*{Special Cases}

We say the process is a {\em pure birth-counting process} if $q_i^-=0$ for all states $i$ and $q_i^+ > 0$ for at least one state $i$ (note that this term should not be confused with a {\em pure-birth process} which is a process with $\mu_i = 0$ for all $i$). Similarly a {\em pure death-counting process} has $q_i^+ = 0$ for {\color{black} all states $i$} and $q_i^- >0$ for at least one state $i$. Further, we call {\color{black} $N_q(0,t]$} a {\color{black} {\em complete pure birth-counting process}} if in addition to being pure birth-counting it also has $q_i^+ =1$ for all states $i$. And similarly a {\color{black} {\em complete pure death-counting process}} is a pure death-counting process that has $q_i^-=1$ for all states $i$. 

Let us see that our finite state space formula, \eqref{eq:RJgenPformula} from Theorem~\ref{thm:main} generalizes the formula, in Theorem~3.1 of  \cite{NazarathyWeiss0336}. Specifically, \cite{NazarathyWeiss0336} studies {\color{black} complete pure death-counting processes}, and by switching the role of births and deaths, one may also apply the results to {\color{black} complete pure birth-counting processes}. The following corollary applies to this case and one may verify that \eqref{eq:equation-for-RRR-as-nazarathy-weiss} agrees with the results in Theorem~3.1 of \cite{NazarathyWeiss0336}.

\begin{corollary} 
\label{cor:Peqbd}
Given a fixed set of birth and death parameters, $\RRR$ is the same for both the {\color{black} complete pure birth-counting process} case, or the {\color{black} complete pure death-counting process} case. In both of these cases,   
\begin{equation}
\label{eq:equation-for-RRR-as-nazarathy-weiss}
\RRR = 1 + 2\lstar \sum_{k=0}^{J-1} \frac{\bigg(P_k - \Lambda_k^-\bigg)\bigg(P_k  - \Lambda_k^+\bigg)}{\pi_{k}\lambda_k},
\end{equation}
where 
\begin{equation}
\label{eq:special-big-Lambda}
\Lambda_k^- = \frac{\sum_{i=0}^{k-1} \pi_i \lambda_i }{\lstar},
\qquad
\text{and}
\qquad
\Lambda_k^+ = \frac{\sum_{i=0}^{k} \pi_i \lambda_i }{\lstar}.
\end{equation}
\end{corollary}

Note that in the {\color{black} complete pure birth-counting process} case, $\Lambda_k$ of \eqref{eq:Lambda} is $\Lambda_k^+$ of \eqref{eq:special-big-Lambda} and similarly in the {\color{black} complete pure death-counting process} case, $\Lambda_k$ is $\Lambda_k^-$. Note also that we take $\Lambda_0^- =0$.

\begin{proof}
It follows from Equation \eqref{eq:lambda-bar-def-first} and \eqref{detbalance} that $\lstar = \sum_{j = 0}^{J-1} \pi_j \lambda_j$ for both the {\color{black} complete pure birth-counting process} and {\color{black} complete pure death-counting process}. However, considering \eqref{eq:Lambda} we see that $\Lambda_k$ differs for these cases and is $\Lambda_k^-$ and $\Lambda_k^+$ for the {\color{black} complete pure birth-counting process} and {\color{black} complete pure death-counting process} respectively.

Now using Theorem~\ref{thm:main}, for the {\color{black} complete pure birth-counting process} case we obtain the following expression for $R_k$ in \eqref{eq:R-k},
\begin{align*}
R_k &=  (P_k -\Lambda_k)\bigg(\frac{\lstar (P_k-\Lambda_k) }{\pi_{k}\lambda_k} + 1 \bigg)\\
&=(P_k -\Lambda_k)\bigg(\frac{\lstar (P_k-\Lambda_{k-1}) }{\pi_{k}\lambda_k} \bigg)\\ 
&=         
\lstar\frac{\bigg(P_k - \Lambda_k^-\bigg)\bigg(P_k  - \Lambda_k^+\bigg)}{\pi_{k}\lambda_k}.
\end{align*}
Similarly, for the {\color{black} complete pure death-counting process} we obtain the same expression on the right hand side. Hence we obtain  \eqref{eq:equation-for-RRR-as-nazarathy-weiss} for both of these cases.        
       
\end{proof}

\section{Proof of the main result}
\label{sec:finite-proof}

We break up the proof into several subsections. First we present a lemma for an explicit inverse of a $J \times J$ matrix. This result maybe of independent interest. Then we construct a renewal-reward process related to $N_q$. Then we compute the asymptotic variance of the renewal reward process based on moments of quantities within a regeneration period. We then carry out a Laplace transform with generating function analysis, to obtain expressions for these moments. Finally, we use the explicit matrix inverse to assemble the pieces together and obtain the result.

\subsection*{An Explicit Matrix Inverse}

Define the following $J \times J$ matrix,
\begin{equation}\label{eq:W}
 \bW =
 \left[ \begin{array}{cccccccc}
 r_1& -\lambda_1 & 0 & . & \cdots & . & 0 \\
  -\mu_2 & r_2 & -\lambda_2 & 0 & \cdots & . & . \\
   0 & -\mu_3 & r_3 & -\lambda_3 & \cdots & . & . \\
   . & 0 & \ddots & \ddots & \ddots & \vdots & \vdots \\
   \vdots & \vdots & \ddots & \ddots & \ddots & \ddots & 0 \\
   . & . & . & \cdot & -\mu_{J-1} & r_{J-1} & -\lambda_{J-1} \\
   0 & . & . & \cdots & 0 & -\mu_J & r_J \\
\end{array} \right],
\end{equation}
where $r_i = \lambda_i + \mu_i$, and $\lambda_i, \mu_i > 0$.

Based on computational exploration with trial and error, we were able to guess an explicit form for the matrix inverse.

\begin{lemma} \label{lem:W}
The matrix $\bW$ is non-singular with elements of the inverse given by,
\begin{equation} 
\label{eq:Wexplicit}
\big(\bW^{-1}\big)_{ij} = 
 \sum_{k=1}^{\min\{i,j\}} \frac{\pi_j}{\pi_k\mu_k}
\qquad(i,j=1,\ldots, J),
\end{equation}
{\color{black} where $\pi_i$ is given by \eqref{eq:21}.}
\end{lemma}
\begin{proof}
Write $\bW = (w_{ij}), \ \bW^{-1} = (u_{ij})$. We leave the reader to
check the case $i=j=J$, namely $(\bW^{-1}\bW)_{JJ} = 1$.  Otherwise, 
\begin{equation*}
(\bW^{-1}\bW)_{11} = u_{11}w_{11} + u_{12}w_{21}
  = \frac{\pi_1}{\pi_1\mu_1}(\lambda_1+\mu_1) 
   - \frac{\pi_2}{\pi_1\mu_1}\mu_2
  = 1;\end{equation*}
for $i=1$ and $j=2,\ldots,J,$
\begin{align*}
(\bW^{-1}\bW)_{1j} &= u_{1,j-1}w_{j-1,j} + u_{1j}w_{jj} + u_{1,j+1} w_{j+1,j}\\
  & = \frac{\pi_{j-1}}{\pi_1\mu_1}\bigg[
  - \lambda_{j-1} + \frac{\lambda_{j-1}}{\mu_j}(\lambda_j+\mu_j)
  - \frac{\lambda_{j-1}\lambda_j}{\mu_j\mu_{j+1}} \mu_{j+1} \bigg]=0;
\end{align*}
for $i,j\ge2, \; i\ne j$,
\begin{align*}
(\bW^{-1}\bW)_{ij} &= u_{i,j-1}w_{j-1,j} + u_{ij}w_{jj} + u_{i,j+1} w_{j+1,j}\\
 &= -u_{i,j-1}\lambda_{j-1} +u_{ij}(\lambda_j+\mu_j) -u_{i,j+1}\mu_{j+1}\\
 &= \bigg(\sum_{k=1}^{\min\{i,j\}} \frac1{\pi_k\mu_k} \bigg)
\pi_{j-1} \bigg[ - \lambda_{j-1} + \frac{\pi_j}{\pi_{j-1}}(\lambda_j+\mu_j)
  - \frac{\pi_{j+1}}{\pi_{j-1}} \mu_{j+1}\bigg]=0;
\end{align*}
and for $2 \le i=j\le J-1$,
\begin{align*}
(\bW^{-1}\bW)_{ii} &= u_{i,i-1}w_{i-1,i} + u_{ii}w_{ii} + u_{i,i+1} w_{i+1,i}\\
 &= -u_{i,i-1}\lambda_{i-1} +u_{ii}(\lambda_i+\mu_i) -u_{i,i+1}\mu_{i+1}\\
 &= \bigg(\sum_{k=1}^{i-1} \frac1{\pi_k\mu_k} \bigg)
\pi_{i-1} \bigg[ - \lambda_{i-1} + \frac{\pi_i}{\pi_{i-1}}(\lambda_i+\mu_i)
  - \frac{\pi_{i+1}}{\pi_{i-1}} \mu_{i+1}\bigg] \\
   &\qquad\qquad\qquad
  + \frac{\pi_i}{\pi_i\mu_i}\Big[(\lambda_i+\mu_i) -
  \frac{\pi_{i+1}}{\pi_i}\mu_{i+1}\Big]\,=\, 0+1 \,=\, 1.
\end{align*}
Combining these cases confirms \eqref{eq:Wexplicit}.
\end{proof}

\subsection*{A Renewal Reward Process}

We exploit regeneration epochs
in the sample path of the BD process $Q$ at exit times $\tilde{t}$
from the state 0, i.e.\ $\{Q(\tilde{t}-)=0, Q(\tilde{t})=1\}$, focusing on
the generic time $X$ that elapses between such epochs and the number $Y$ of
birth and death epochs that are counted (i.e.\ are not
thinned) by the counting function {\color{black} $N_q(\cdot)$}
during this time-interval of length $X$.  Then {\color{black} $N_q(\cdot)$} has asymptotic
behaviour the same as a renewal--reward process $C(\cdot)$ for which 
$\lim_{t\to\infty} t^{-1} \var C(t)$ is known (\cite{brown1974soa} and
\eqref{eq:momlimits}--\eqref{eq:RRRformula} below).  This
derivation shows that the initial distribution of $Q(0)$ plays no part
in the theorem when $J$ is finite, i.e.\ the result depends only on the
existence and properties of a stationary distribution $\{\pi_i\}$, and thus holds both for a stationary and non-stationary version of the process.

Let $\calT$ be the set of regeneration times in $\Rpos$ when
$Q(\cdot)$ exits the state 0, i.e.\
\begin{equation} \label{eq25} 
  \calT = \{\tilde{t} : Q(\tilde{t}-)=0,\,Q(\tilde{t})=1 \}.
\end{equation} 
Because $\pi_0>0$, $\calT$ is a countably infinite sequence which a.s.\ has no
finite limit point; enumerate
$\calT$ as the increasing sequence $\{T_i\}$.  Denote
\begin{equation} \label{eq26}
  X_i = T_i - T_{i-1} \quad (i=1,2,\ldots),
\end{equation} 
and write $Y_i = N_q(T_{i-1}, T_i]$ for the number of transitions counted (after thinning)
during the half-open time interval shown.  Write also
\begin{equation}
\label{eq:n-tilde-pm-ij}
  \tilde{n}^\pm_{i,j} := \card
 \big\{t \in (T_{i-1},T_i] : Q(t-) =j, \ Q(t)=j\pm 1\big\},
\end{equation}
respectively, for which almost surely for every $i$,
$\tilde{n}^-_{i,0} = 0$, $\tilde{n}^+_{i,0}=1=\tilde{n}^-_{i,1}$,
 $\tilde{n}^+_{i,j}=\tilde{n}^-_{i,j+1} \in \{0,1,\ldots\}$   
for $j=1,\ldots,J-1$, and    
\begin{equation} \label{eq27}
  Y_i := \sum\nolimits_{j=0}^{\Js-1} \mbox{Bin}(\tilde{n}^+_{i,j},\,q^+_j)
  + \sum\nolimits_{j=1}^\Js \mbox{Bin}(\tilde{n}^-_{i,j},\,q^-_j)
  =: Y_i^+ + Y_i^-\,,
\end{equation}
say, where $\Bin(k,p)$ denotes a binomial $(k,p)$ r.v.
Because the process $Q(\cdot)$ has the strong Markov property, the sequence
$\big\{(X_i,Y_i), i \ge 1 \big\}$ is a sequence of i.i.d.\ random vectors.
Let $(X,Y)$ denote a generic element of this sequence.  Define 
\begin{equation} \label{eq:Ndecomp}
  N_\calT(t) := \sup \Big\{ \ell : \sum\nolimits_{i=1}^\ell X_i \le t \Big\},
   \qquad C(t) = \sum\nolimits_{i=1}^{N_\calT(t)} Y_i.
\end{equation}
Then $C(\cdot)$ is a {\sl renewal--reward process}; $C(t)$ is the number of
BD transitions counted during regeneration intervals $X_i$ wholly contained in
$(0,t]$, and \eqref{eq:Ndecomp} exhibits $Y_i$ as
the sum of counted BD transitions accumulating during each interval $X_i$ and
finally included in $C(\cdot)$ at the regeneration epochs occurring by $t$.  
From our definitions,
\begin{equation} \label{eq:NCrelation}
 N_q(0,t] = C(t) + \widetilde{N}_q(t),
\end{equation}
where $\widetilde{N}_q(t)$ is the number of counted BD transitions occurring
in the time interval $(T_{N_q(t)-1}, t]$.  Take first and second moments of
both sides of \eqref{eq:NCrelation}, divide by $t$ and let $t\to\infty$.  Then
\begin{equation}\label{eq:momlimits}
 \begin{aligned}
 \lim_{t\to\infty}\frac{\E\big[N_q(0,t]\big]}{t}
  &=\lim_{t\to\infty}\frac{\E\big[C(t)\big]}{t},\\
   \lim_{t \to \infty} \frac{\var N_q(0,t]}{t}
  &= \lim_{t \to \infty} \frac{\var C(t)}{t},
 \end{aligned}
\end{equation}
where we have used the fact that
$\E\big[\big(\widetilde{N_q}(t)\big)^2\big] < \infty$,
hence  $\cov\big(C(t), \widetilde{N}_q(t)\big)=O(\sqrt{t}\,)$. 
The two limit relations at
\eqref{eq:momlimits} imply that
\eqan{
\RRR = \lim_{t\to\infty} \frac{\var C(t)}{\E\big[C(t)\big]},
 }
and thus we may study $C(t)$ in place of $N_q(0,t]$.

A representation of the asymptotic variance of renewal--reward processes in
terms of moments of a generic element $(X,Y)$ is long known
(see \cite{brown1974soa}), from which
work it follows that when $X$ and $Y$ have finite second-order moments,
\begin{equation} \label{eq:RRRformula}
\RRR = \frac{ \var\big(Y - \E[Y] (X/\E[X])\big)}{\E[Y]}  
  = \frac{\E(Y^2)- 2R\, \E(XY) + R^2\, \E(X^2)}{\E(Y)},
\end{equation}
where $R = \E(Y)/\E(X)$, showing that $\RRR$ is determined by $\E[X]$,
$\E[Y]$, $\E[XY]$, $\E[X^2]$ and $\E[Y^2]$.
In \eqref{eq:RRRformula}, $Y$ is the number of counted births and
deaths in a generic regeneration interval of duration $X$, being a generic
recurrence time of the state 0 of the stationary Markov process $Q$, so
$\E(X)=1/(\pi_0\lambda_0)$.  Births and deaths occur in $Q$ at rate $\lstar_\star$,
and $N_q(\cdot)$ counts a proportion $\varpi$ of these (see \eqref{countratio})
at rate $\lstar = \varpi\lstar_\star$, so for the pair $(X,Y)$, 
 the mean count $\E(Y)$ equals
 $\varpi \lstar_\star \E(X)$.   Thus, 
$R = \E(Y)/\E(X) = \lstar$ and $\E(Y)=\lstar/(\pi_0\lambda_0)$. 

To compute these moments we study the evolution of the $\bbbJ$-valued strong
Markov process $Q$ in more detail over a generic regeneration interval.
Consider the r.v.s, for $j, k=1,\ldots,J$, 
\begin{align}\label{regintXrvs}
 \tau_k &:= \inf \{t>0 : Q(t) = 0 \mid Q(0)=k\},\nonumber\\
 \breve{n}^\pm_{j,k} &:=
\card\big\{t\in(0,\tau_k] : Q(t-) =j,\, Q(t)=j\pm 1 \mid Q(0)=k\big\},\\
\Nbp(\tau_k) &:= \big(\textstyle\sum_{j=1}^{\Js-1}\Bin(\breve{n}^+_{j,k},q^+_j)
 + \sum_{j=1}^\Js \Bin(\breve{n}^-_{j,k},q^-_j)\mid Q(0)=k\big),\nonumber
\end{align}
where for given $k$ and $\breve{n}^\pm_{j,k}$, each of the sequences of
binomial r.v.s $\{\Bin(\breve{n}^\pm_{j,k},q^\pm_j): j=1,\ldots,J\}$ has
mutually independent elements.  (Strictly, each of the r.v.s defined in
\eqref{regintXrvs} is a functional defined on the sample path $\{Q(u) :
0\le u\le\tau_k\}$, where given the sample path, each of $\breve{n}^\pm_{j,k}$
records the numbers of births or deaths resp., and then, conditional on these
numbers, $N_J(\tau_k)$ counts the number of births or deaths included in the
two components of $N_q(\cdot)$ at \eqref{eq:19StanleySt} over the regenerative
interval $X_i$ of which
$\tau_k$ is part; so, {\em conditional} on the sample path of $Q(\cdot)$, hence
also $\{\breve{n}^\pm_{j,k}\}$,
$N_J$ has the representation as shown
in terms of independent Binomial r.v.s.). 
Note the distinction between the counting
r.v.s $\breve{n}^\pm_{j,k}$ here and $\tilde{n}^\pm_{i,j}$ in \eqref{eq:n-tilde-pm-ij}.

Any regeneration interval $X_i$ as at \eqref{eq26} starting at time $T_{i-1}$
consists first of an interval which we call the busy period, and then of an interval which we call the idle period. The busy period is distributed as $\tau_1$ and the idle period is distributed as an exponential random variable, say $\sigma_0$, with rate $\lambda_0$. With this,
%
$X_i =^d \tau_1 + \sigma_0 $, and $Y_i =^d  N_J(\tau_1) +\Bin(1,q_0^+)$.  In these representations,
$\sigma_0$ is independent of both $\tau_1$ and $N_J(\tau_1)$, and
$\Bin(1,q_0^+)$ and $N_J(\tau_1)$ are independent; of course $\tau_1$ and
$N_J(\tau_1)$ are in general dependent.

For $h=1,2$, define the moments
\begin{align} \label{BPdefns}
\tau^{(h)}_k &= \E[\tau_k^h] \,=\, \E[\tau_k^h \mid Q(0) = k],
\nonumber\\
n^{(h)}_k &=
 \E\big[\big(\Nbp(\tau_k)\big)^h \mid Q(0)=k\big], \\
c_k^{(1)} &= 
\E [ \tau_k \,\Nbp(\tau_k) \mid Q(0)=k].
\nonumber
 \end{align}
Observe that $\tau^{(1)}_0=n^{(1)}_0=c^{(1)}_0=\tau^{(2)}_0=n^{(2)}_0=0$. Below
we develop expressions for the moments $\tau_1^{(1)}$, $n^{(1)}_1$, $c^{(1)}_1$,
$\tau^{(2)}_1$ and $n^{(2)}_1$, using these to find the moments $\E[XY]$, $\E[X^2]$, and $\E[Y^2]$ for \eqref{eq:RRRformula}. Recalling the expressions for $\E[X]$, $\E[Y]$ we have
\begin{equation} \label{eq:EXEXXEXY}
\begin{aligned}
 \E[X]&= \tau_1^{(1)} + \frac1{\lambda_0} = \frac1{\pi_0\lambda_0},\\
\E[Y] &= n_1^{(1)} + q_0^+ \,=\, \frac{\lstar}{\pi_0 \lambda_0}.
\end{aligned}
\end{equation}
And thus,
\begin{equation}
\label{eq:tau-and-n-ones}
\tau_1^{(1)} = (1-\pi_0)/ (\pi_0\lambda_0),
\qquad
\text{and} 
\qquad
n_1^{(1)} = (\lstar - \pi_0 \lambda_0 q_0^+ )/(\pi_0 \lambda_0).
\end{equation}
Now continuing with the busy period and idle period relationship we have,
\begin{equation} \label{eq:EXEXXEXY-second}
\begin{aligned}
 \E[X^2]&= \tau_1^{(2)} + \frac{2\tau_1^{(1)}}{\lambda_0} + \frac2{\lambda_0^2}
  \,=\, \tau_1^{(2)} + \frac2{\pi_0\lambda_0^2}\,,\\
 \E[X Y] &= c_1^{(1)} + \frac{n_1^{(1)}}{\lambda_0} + q_0^+\Big(\tau_1^{(1)}+
\frac1{\lambda_0}\Big) \,=\, c_1^{(1)} + \frac{n_1^{(1)}}{\lambda_0}
   + \frac{q_0^+}{\pi_0\lambda_0}\,,\\
   \E[Y^2] &= n_1^{(2)} + 2n_1^{(1)} q_0^+ + q_0^+\,.
\end{aligned}
\end{equation}
Now using \eqref{eq:RRRformula} and \eqref{eq:tau-and-n-ones} we have,
\begin{equation}
\label{eq:ey-d-all-together}    
\E(Y) \RRR 
=
n_1^{(2)} - 2\lstar c_1^{(1)} + \lstar^2 \tau_1^{(2)} +  q_0^+\Big(1 - 2 q_0^+ + 2\pi_0 \E(Y) \Big).
\end{equation}

\subsection*{A Transform Based Calculation}

To use \eqref{eq:ey-d-all-together}, we now find a representation for the moments $n_1^{(2)}$, $c_1^{(1)}$, and $\tau_1^{(2)}$ using first step analysis. Note that this process also yields expressions for the first moments $\tau_1^{(1)}$ and $n_1^{(1)}$ which naturally agree with those computed above. For such a first step analysis, the r.v.s at \eqref{regintXrvs} are amenable to study via a backwards
decomposition in terms of the BD process $Q$ as follows.  Given $Q(0)=k$ for
some $k=1,\ldots,J$, define $\phi_k(s,z) = \E[\e^{-s \tau_k} z^{N_J(\tau_k)}
\mid Q(0)=k]$.  The first transition in $Q$ from $Q(0)=k$ occurs after the
random time $\sigma_k$ which is exponentially distributed with mean
$1/r_k := 1/(\lambda_k+\mu_k)$, after which $Q= k\pm 1$
(at rate $\lambda_k$ for +1, and $\mu_k$ for $-1$), and $N_J(\tau_k)$
increases by 1 with probability $q_k^\pm$ depending on $\pm1$ change in $Q$.
Using indicator r.v.s $I(q)=1$ with probability $q$, $=0$ otherwise, these
relations can be written as
\begin{equation}\label{eq:backwdpgf}
   (\tau_k, N_J(\tau_k)) = 
  \Bigg\{
\begin{aligned}
&\big(\sigma_k + \tau_{k+1}, \,I(q_k^+) + N_J(\tau_{k+1})\big)
   \;\;\text{with prob. } \lambda_k/r_k,\\
&\big(\sigma_k + \tau_{k-1}, \,I(q_k^-) + N_J(\tau_{k-1})\big)
   \;\;\text{with prob. } \mu_k/r_k.\\
\end{aligned}
  \end{equation}
Using generating functions and $\E[\e^{-s\sigma_k}] = r_k/(s+r_k)$ we deduce that
\begin{equation} \label{phipgf}
  (s+r_k)\phi_k(s,z) = 
   \lambda_k(z q_k^++1-q_k^+)\phi_{k+1}(s,z) 
       + \mu_k(z q_k^- + 1 - q_k^-)\phi_{k-1}(s,z).
  \end{equation} 

Differentiation in $s$ and $z$ and setting $(s,z)=(0,1)$ yields the relations   
for $k = 1,\ldots,J$ (we can and do include $k=1$ and $k=J$ because
$\lambda_J=0$ and $\tau^{(1)}_0=n^{(1)}_0=c^{(1)}_0=\tau^{(2)}_0=n^{(2)}_0=0$):
\begin{align} \label{eq:moms}
r_k\tau^{(1)}_k - \lambda_k \tau^{(1)}_{k+1}
  - \mu_k \tau^{(1)}_{k-1} &= 1, \nonumber\\
r_k n^{(1)}_k - \lambda_k n^{(1)}_{k+1}
 - \mu_k n^{(1)}_{k-1} &= \lambda_k q^+_k + \mu_k q^-_k,
\nonumber\\
r_k c^{(1)}_k - \lambda_k c^{(1)}_{k+1}
  - \mu_k c^{(1)}_{k-1} &=  n^{(1)}_k + \lambda_k q^+_k \tau^{(1)}_{k+1} 
  + \mu_k q^-_k \tau^{(1)}_{k-1}, \\
r_k \tau^{(2)}_k - \lambda_k \tau^{(2)}_{k+1}
  - \mu_k \tau^{(2)}_{k-1} &=  2 \tau^{(1)}_k, \nonumber\\
r_k n^{(2)}_k - \lambda_k n^{(2)}_{k+1} - \mu_k n^{(2)}_{k-1}
&= 
\lambda_k q_k^+ + \mu_k q^-_k +
 2(\lambda_k q^+_k n^{(1)}_{k+1} + \mu_k q^-_k n^{(1)}_{k-1}).\nonumber
\end{align}
In these five equations the left-hand sides are of the same form.  Moreover
this form is familiar from the study of BD equations because they come from a
backward decomposition of the first-passage time r.v.s $\{\tau_k\}$ and
functionals $N_J(\tau_k)$, all defined in terms of the underlying BD process
$Q$. 

\subsection*{Assembling the Components for the Final Expression}

We now represent \eqref{eq:moms} in matrix form using 
$J$-vectors (taken as columns), for $h=1,2,$
$\btau^{(h)} = {(\tau_1^{(h)}, \ldots,\tau_J^{(h)})}$,  
$\bn^{(h)} = {(n_1^{(h)},\ldots, n_J^{(h)})}$ and
$\bc^{(1)} = {(c_1^{(1)},\ldots,c_J^{(1)})}$ and using the $J \times J$ matrix $\bW$ from \eqref{eq:W}.

With this notation, the left-hand sides of \eqref{eq:moms} are expressible as $\bW\bx$ for
the relevant $J$-vector $\bx={(x_1,\ldots,x_J)}$.   
To describe the right-hand sides of the equations at \eqref{eq:moms} in matrix
form write $\ba\smult\bb$ for the $J$-vector consisting of the element-wise
products of the $J$-vectors $\ba = (a_1, \ldots, a_J)$ and $\bb = (b_1, \ldots, b_J)$, $\bb_{[-]}$ for the $J$-vector
$(0, b_1,\ldots,b_{J-1})$,
$\bb_{[+]}$ for the $J$-vector $(b_2,\ldots,b_J, 0)$
and $\bpi = (\pi_1,\ldots,\pi_J)$ so that
$\bpi^\top \bone + \pi_0 = 1$.
Write also $\blambda = (\lambda_1,\ldots,\lambda_J)$,
 $\bmu = (\mu_1,\ldots,\mu_J)$ and ${\bq^\pm} =
(q^\pm_1,\ldots,q^\pm_J)$.

The notation $\ba\circ\bb$ is useful and has $(\ba\circ\bb)\circ\bc
= \ba\circ(\bb\circ\bc) = \ba\circ\bb\circ\bc$ and 
$\bpi^\top(\ba\circ\bb) = \bone^\top(\bpi\circ\ba\circ\bb)$,
but has the drawback that, for a general
$J\times J$ matrix $\bA$, $\bA (\ba\circ \bb) \ne (\bA \ba)\circ\bb$. We also define the $J\times J$-matrix
$\bV =: (v_{ij})$ for $i,j\in\{1,\ldots,J\}$ with $v_{ij}=0$ except for its super- and sub-diagonal elements $v_{i,i+1} = \lambda_i q_i^+$ ($i=1,\ldots,J-1$) and $v_{i,i-1} = \mu_i q^-_i$ $(i=2,\ldots,J)$. This matrix helps facilitate algebraic manipulation.

With these definitions the five equations \eqref{eq:moms} are expressible as,
\begin{align}
\label{eq:moms2}
 \bW\btau^{(1)} &= \bone, \nonumber\\
\bW\bn^{(1)} &= \blambda\smult\bq^+ + \bmu\smult\bq^-,
 \nonumber\\
 \bW\bc^{(1)} &= \bn^{(1)} + (\blambda\smult\bq^+)\smult\btau^{(1)}_{[+]} + 
  (\bmu\smult\bq^-)\smult\btau^{(1)}_{[-]}\\
  &= \bn^{(1)} + \bV \btau^{(1)}
  ,   \nonumber\\
 \bW\btau^{(2)} &= 2\btau^{(1)},
 \nonumber\\
 \bW\bn^{(2)} &= 
 \blambda \smult \bq^+ + \bmu \smult \bq^- +  
 2 \big( \blambda\smult\bq^+ \smult \bn^{(1)}_{[+]}
  ~+~ \bmu \smult \bq^- \smult\bn^{(1)}_{[-]}\big) \nonumber \\
    &= \blambda \smult \bq^+ + \bmu \smult \bq^- + 2 \bV \bn^{(1)}
      \nonumber,
\end{align}
and hence by multiplying each equation in \eqref{eq:moms2} by $\bW^{-1}$ we have a representation of the desired vector. For example $\btau^{(1)} = \bW^{-1} \bone$ and $\bn^{(1)} = \bW^{-1} (\blambda\smult\bq^+ + \bmu\smult\bq^-)$. In both of these cases, by left multiplying the expression with $\be_1=(1,0,\ldots,0)$ we retrieve the previously obtained expressions in \eqref{eq:tau-and-n-ones}. Observe also that 
\begin{equation}
\label{eq:n2-in-terms-of-n1}
\bn^{(2)} = \bn^{(1)} + 2 \bW^{-1} \bV \bn^{(1)}.
\end{equation}
%






We now return to \eqref{eq:ey-d-all-together} and use \eqref{eq:moms2}, \eqref{eq:n2-in-terms-of-n1}, and $\bW^{-1}$ to develop an expression for the sum of the first three terms of \eqref{eq:ey-d-all-together}. Namely, 

 \begin{equation} \label{eq:DJnctau}
 \begin{aligned}
n_1^{(2)} - 2\lstar c_1^{(1)} + \lstar^2 \tau_1^{(2)}
 &= n_1^{(1)} + 2 \be_1^\top \bW^{-1} \bV \bn^{(1)}
  - 2\lstar\be_1^\top \bW^{-1}(\bn^{(1)} + \bV \btau^{(1)})
  + 2 \lstar^2\be_1^\top \bW^{-1} \btau^{(1)} \nonumber \\
 &= n_1^{(1)} + 2 \be_1^\top\bW^{-1} \big( \bV \bn^{(1)} - \lstar \bn^{(1)} 
 - \lstar \bV \btau^{(1)} + \lstar^2 \btau^{(1)} \big) \nonumber \\
 &= n_1^{(1)} + 2\be_1^\top\bW^{-1}(\bV -\lstar\bI)(\bn^{(1)} - \lstar\btau^{(1)})
 \nonumber\\
 &= n_1^{(1)} + \frac{2}{\pi_0\lambda_0} \bpi^\top
  (\bV - \lstar\bI) \bW^{-1}\big(\blambda\smult\bq^+ + \bmu\smult\bq^-
- \lstar\bone\big).
\end{aligned}
\end{equation}
In the last step we used Lemma~\ref{lem:W} to represent $\be_1^\top\bW^{-1}$ as $\bpi^\top/(\pi_1 \mu_1)$ and then use \eqref{detbalance}.


Now returning to \eqref{eq:ey-d-all-together} and we have
\begin{equation} \label{eq:DJAB}
 \RRR - 2\pi_0 q_0^+ - \frac{q_0^+(1 - 2q_0^+) + n_1^{(1)}}{\E(Y)} 
= \frac2{\lstar} \, \ba^\top \bW^{-1}  {\bb}  \,,
\end{equation}
where $\ba^\top := \bpi^\top (\bV-\lstar\bI)$ and
$\bb := \blambda\smult\bq^+ + \bmu\smult\bq^- - \lstar\bone$. Considering the bi-linear form on the right hand side of \eqref{eq:DJAB} we now have,
%
%
\begin{equation*}
  \ba^\top \bW^{-1}\bb 
  = \sum_{i=1}^J \sum_{j=1}^J
  \sum_{k=1}^{\min\{i,j\}} a_i \frac{\pi_j}{\pi_k \mu_k} b_j
  =
  \sum_{k=1}^J\frac1{\pi_k\mu_k}
   \sum_{i=k}^J a_i \sum_{j=k}^J \pi_j b_j,
\end{equation*}
where the first step follows from Lemma~\ref{lem:W} and the second step is due to elementary manipulation, 
$\sum_{i=1}^J \sum_{j=1}^J \sum_{k=1}^{\min\{i,j\}}
= \sum_{k=1}^J \sum_{i=k}^J \sum_{j=k}^J$.


We now define the partial sums $A_k := \sum_{j=1}^k a_j$ and
$B^\pi_k := \sum_{j=1}^k \pi_j b_j$ for $k=0,\ldots,J$. Noting the empty sum case, $A_0 = B^\pi_0 =0$, we now have
\begin{align} 
\label{eq:BCsum}
   \ba^\top \bW^{-1}\bb 
  = \sum_{k=1}^J
   \frac{ (B^\pi_J - B^\pi_{k-1})(A_J - A_{k-1})}{\pi_k\mu_k}
  = \sum_{k=0}^{J-1}
   \frac{ (B^\pi_J - B^\pi_k)(A_J - A_k)}{\pi_k\lambda_k}\,.
\end{align} 

Now using the expression for $\lambda_k$ in  \eqref{eq:little-lambda-k-def} as well as the definitions of $\Lambda_k$ and $P_k$, with simple manipulation we can represent $A_k$ and $B_k^\pi$ as,
\begin{equation*}
 A_k = \begin{cases}
\lstar_k - \pi_1\mu_1 q^-_1 - \pi_0\lambda_0 q^+_0 - \lstar(P_k - \pi_0) - \pi_k\lambda_kq_k^+
+ \pi_k\lambda_k q^-_{k+1},
& k=1,\ldots,J-1, \\
\noalign{\smallskip}
\lstar\pi_0 - \pi_1\mu_1 q^-_1 - \pi_0\lambda_0 q^+_0 , & k=J,
\end{cases}
\end{equation*}
and
\begin{equation*}
B_k^\pi = \lstar(\Lambda_k - P_k) + \pi_0(\lstar - \lambda_0 q_0^+),
\qquad
\text{for}
\qquad 
k=1,\ldots,J,
\end{equation*}
where we observe that $B_J^\pi = \pi_0(\lstar - \lambda_0 q_0^+)$.

Thus in the summation on the right of \eqref{eq:BCsum}, the denominator factors for $k=1,\ldots,J$, can be represented as,
\[
A_J - A_k
= \lstar\big(P_k - \Lambda_k\big) + \pi_k\lambda_k(q_k^+ - q_{k+1}^-)
\qquad
\text{and}
\qquad
B^\pi_J - B^\pi_k = \lstar\big(P_k- \Lambda_k\big).
\]

With these expression and \eqref{eq:BCsum}, \eqref{eq:DJAB}, and $\E(Y) = \lstar/(\pi_0 \lambda_0)$, after some standard manipulation we obtain the expression in the result, \eqref{eq:RJgenPformula}.  This proves Theorem~\ref{thm:main}.

\section{Examples and Illustrations}
\label{sec:exx}

We first present examples of finite state queueing systems where there is modeling value for understanding the asymptotic index of dispersion. We then present examples of infinite systems, based on the formal expression in \eqref{eq:Dinfbd} with a purpose of exploring the validity of that expression. {\color{black} Note that not all examples are related to BRAVO. In particular, in some of the examples below, the process $N_q$ is not related to the output counting process of a queueing system.}

\subsection*{Finite State Queueing Models}

{\color{black} Let us consider two examples arising from finite state queueing models. The first example illustrates the existence of a BRAVO phenomenon in queues with reneging, while the second example studies a process not related to the output of a queue or the BRAVO effect. Note that} the departure process of finite state queues without reneging (M/M/1/K, M/M/s/K, M/M/K/K) was already analyzed extensively in \cite{NazarathyWeiss0336} since the main result there (appearing as Corollary~\ref{cor:Peqbd} in this paper) covers that case. {\color{black} For such queues, BRAVO was observed.} For this current work, we were originally interested in the variance rate of thinned death processes because the output of a queue with reneging has such a structure. That type of process is not covered in \cite{NazarathyWeiss0336} since it requires state dependent thinning. 


\begin{figure}[h]
\begin{center}
\includegraphics[width=4.2in]{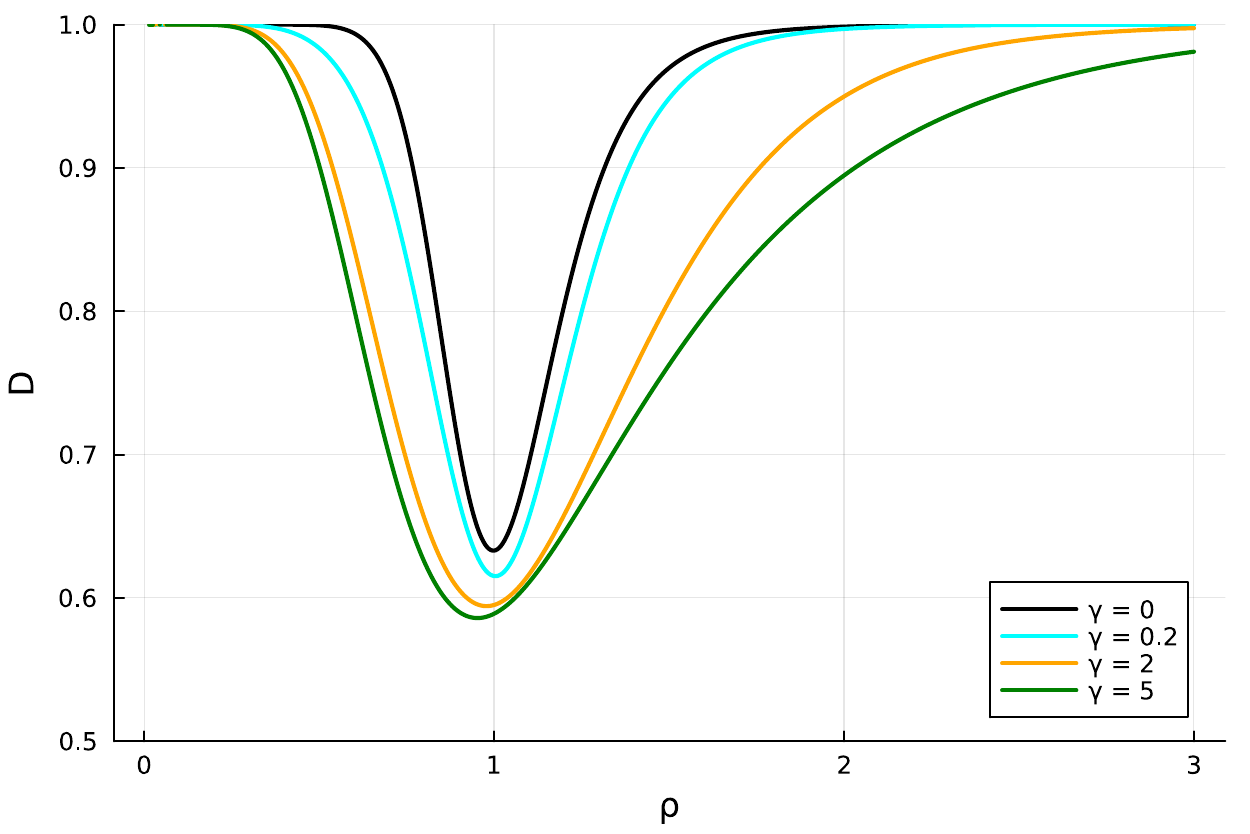}
\captionsetup{width=0.8\linewidth}
\caption{
The asymptotic index of dispersion for M/M/$s$/K$+$M systems with $K=20$, $s=10$, and $\mu = 1$. When the offered load {\color{blue} $\rho \approx 1$}, the asymptotic variability of the output process is reduced in comparison to the Poisson process ($\RRR = 1$) case.
\label{fig:renegFig}}
\end{center}
\end{figure}
 
\begin{example}
{\bf BRAVO for a queue with reneging.}
Consider the many-server Poisson system M/M/$s$ with
reneging and a finite buffer of size $K$.
It is sometimes described as an M/M/$s$/K$+$M system.
In terms of a birth--death
process on finite state
space, we set the state space upper limit $J=K$, the number of servers $s \le K$, the arrival rate $\lambda > 0$, the service rate (per server) $\mu > 0$, and the abandonment rate $\gamma \ge 0$. With these, the BD and thinning parameters are,
\[
\lambda_i=\lambda,
\qquad
\mu_i = \mu \min(i,s) + \gamma \max(i-s,0),
\qquad
q_i^- = \frac{\mu \min(i,s) }{\mu_i}\,,
\qquad
q_i^+ = 0.
\]

Setting $\rho = \lambda/(\mu s)$,   Figure~\ref{fig:renegFig} presents the asymptotic index of dispersion as a function of $\rho$ for various values of $\gamma$. This is for $K=20$, $s=10$ and $\mu=1$. Indeed it is evident that a BRAVO effect appears also in this type of system. {\color{black} It is evident that for $\gamma = 0$ the minimizer is at $\rho = 1$ (this case was also covered in \cite{NazarathyWeiss0336}), whereas for $\gamma >0$ the minimizer deviates slightly from $\rho=1$. There is room for an asymptotic analysis here, similar in nature to \cite{daley2013bravo}, considering the effects of $\gamma$, as well as $K \to \infty$ and $s \to \infty$. We leave such an asymptotic analysis for further research, yet observe at this point that indeed a form {\color{blue} of} the BRAVO effect holds.}

\end{example}

\begin{example}
{\bf Finite population interaction upon arrival.} 
{\color{black} This example illustrates a situation where counting thinned arrivals (as opposed to departures) maybe of interest. We note that this example is not about BRAVO. In particular, this example does not illustrate existence/non-existence of a BRAVO effect since it does not focus on the departure process of a queue. Nevertheless, the resulting curve of $\RRR$ exhibits interesting behavior. Note that we informally refer to this example as ``counting animal fights in a billabong''.}

Consider a finite population queueing system with $J$ individuals in the population, arriving {\color{black} to}, and departing {\color{black} from} an unlimited service (infinite server) queue (billabong). An example is a small {\color{black} finite} population of a certain species occasionally meeting at a fresh water source (billabong). As an additional individual arrives to the water source, there is a possibility of it interacting {\color{black}(e.g. fighting)} with one of the other individuals already there, or not. This probability increases as the number of individuals at the water source increases. {\color{black} With this, we are interested in the counting process of the number of interactions (fights) occurring at the water source.}

For this situation we set the BD and thinning parameters as,
\[
\lambda_i=(J-i)\lambda,
\qquad
\mu_i = i \, \mu,
\qquad
q_i^- = 0,
\qquad
q_i^+ = \frac{i}{i+1},
\]
where we note that $\lambda_i$ and $\mu_i$ in this form capture the finite population and the unlimited service. {\color{black} Further, $q_i^+$ captures the increasing probability of an interaction (fight) upon arrival.}

Figure~\ref{fig:finite-pop} presents the asymptotic index of dispersion as a function of $\lambda$ when $\mu=1$, for various population sizes, $J$. Observe that the asymptotic index of dispersion can be both greater and smaller than unity. Also note (not in figure), that as $\lambda \to \infty$ the asymptotic index of dispersion rises to $1$.

\end{example}

\begin{figure}[h!]
\begin{center}
\includegraphics[width=4.2in]{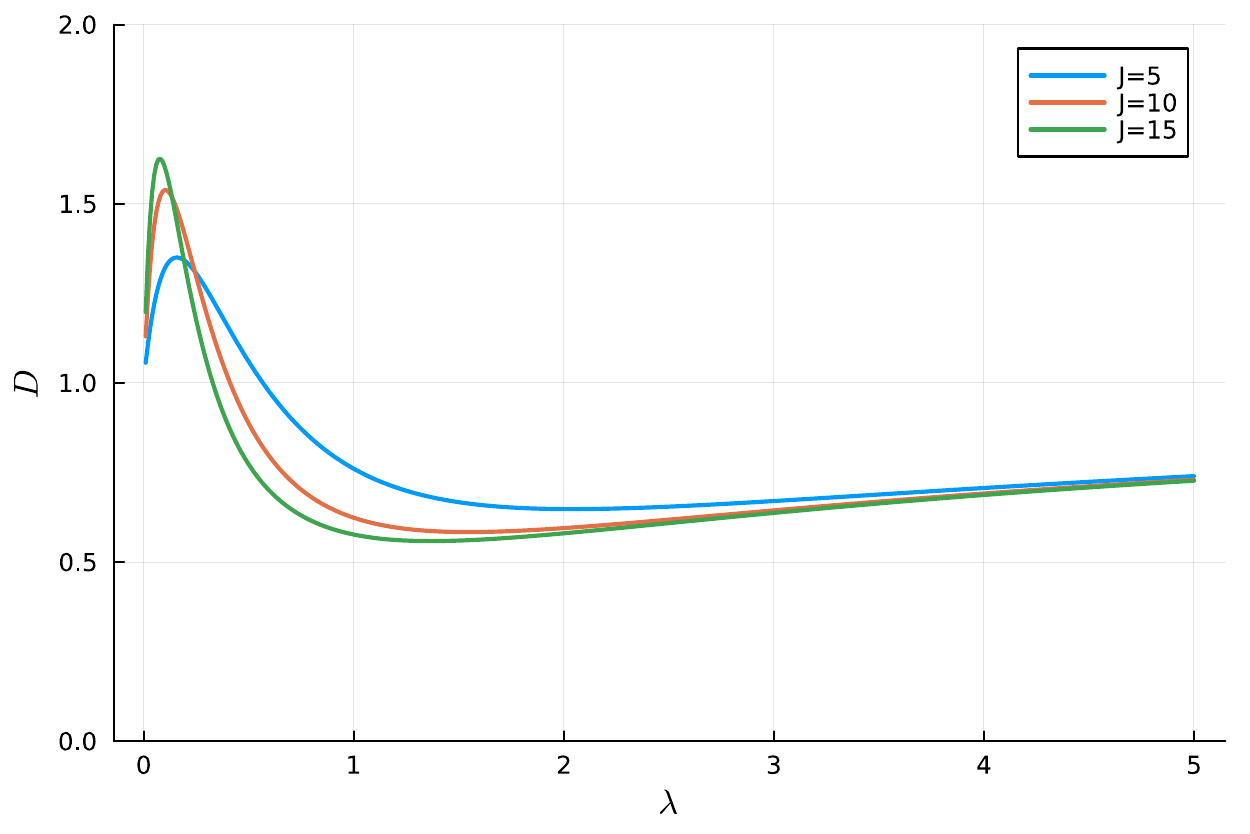}
\captionsetup{width=0.8\linewidth}
\caption{The asymptotic index of dispersion for a counting process of interactions upon arrival to a finite population {\color{black} infinite-}service system.
\label{fig:finite-pop}}
\end{center}
\end{figure}


\subsection*{Infinite State Calculations}

The calculations in this section rely on the formal expression \eqref{eq:Dinfbd} for the infinite state space case. We have not presented a proof for the validity of \eqref{eq:Dinfbd}. In particular, characterization of the conditions for which  \eqref{eq:Dinfbd} is valid is a matter requiring further research. Nevertheless, we believe that the following infinite state space calculations are insightful as they hint at the validity of \eqref{eq:Dinfbd}.

\begin{example}
{\bf Renewal process of positive recurrent M/M/1 busy cycles.} 
Assume $q^+_i \equiv 0$ for all $i$, $q^-_1=1$
and $q^-_i \equiv 0 $ for $i\ge 2$. This is a pure death-counting process. The BD process for a simple queue
has $\lambda_i=\lambda$, $\mu_i=\mu$ and $\rho := \lambda/\mu < 1$.  Further,
$\pi_k = (1-\rho)\rho^k$ and $1-P_k = \rho^{k+1}$. A busy cycle is a time
interval between $Q$ exiting 0. Then \eqref{eq:Dinfbd}
simplifies to,
\begin{equation}
\label{eq:mm1BPrenewal}
\RRR_\infty = 1-2(1-\rho)\rho + 2 \rho^2 \sum_{k=1}^\infty \rho^k
= 1- 2 \rho \frac{1-2\rho}{1-\rho}.
\end{equation}
Interestingly, for $\rho \in (0,\frac12)$, $\RRR < 1$, for $\rho=\frac12$,
$\RRR=1$ and for $\rho \in (\frac12,1)$, $\RRR>1$. The minimum is at
$\rho = \frac12(2- \sqrt{2}) \approx 0.292893$ at which point, $\RRR=4 \sqrt{2}-5 \approx 0.656854$. 
The formula of \eqref{eq:mm1BPrenewal} indeed agrees with classic queueing
theory results based on the so-called M/G/1 busy period functional equation
(see e.g.\ the computation in the proof of
\cite[Proposition 6]{hautphenne2013second}). In this case, using the fact that
$X=B+I$, where $B$ is the busy period r.v.\ and $I$ is the (independent) idle
period r.v., exponentially distributed with rate $\lambda$, it is easy to
obtain,
\begin{align*}
\E[X] &= \frac{1}{\mu} \frac{1}{1-\rho} + \frac{1}{\lambda} = \mu^{-1}\frac{1}{\rho (1-\rho)},\\
\E[X^2] &= \frac{1}{\mu^2}\frac{2}{(1-\rho)^3} +
2\Big(\frac{1}{\mu} \frac{1}{1-\rho}\Big) \Big( \frac{1}{\lambda}\Big) + \frac{2}{\lambda^2}
= 2\mu^{-2} \frac{1-2 \rho(1-\rho)}{\rho^2(1-\rho)^3}.
\end{align*} 
Combining the above gives $\E[X^2]/(\E[X]^2)-1$ consistent with \eqref{eq:mm1BPrenewal}.

%
\end{example}

\begin{example} {\bf Constant birth rates, {\color{black} complete pure death-counting process}.} 
Consider the infinite state space case where $\lambda_i=\lambda$ (constant)
for all $i$, $q^+_i = 0$ (all $i$), $q_i^- \equiv1$, 
$\mu_i = \min\{i,s\}\mu$ and $\rho=\lambda/(s\mu)<1$.
This is the case for output of a stable M/M/$s$ queue.
By reversibility, \cite{kellyBook}, it is known that in
the stationary case the output process is Poisson and hence has $\RRR = 1$.
Indeed using \eqref{eq:Dinfbd}, $\lstar =\lambda$ and
$\Lambda_k=P_{k-1}$ so every product in the numerators in the sum 
preceding \eqref{eq:Dinfbd} is zero, and $\RRR=1$.


Suppose now that $q_i^-=q$ (all $i=1,2,\ldots$) for some $q$ in $(0,1)$.
Then the thinning of departures is independent of the state and
hence the output process is still a Poisson process. Indeed repeating
the calculations (observing that $\lstar = q \lambda$) we again
obtain $\RRR=1$.
\end{example}

\begin{example}
\label{example:}
{\bf M/M/1: Thinning and counting both the input and output processes at constant probabilities.}
Consider again a stable M/M/1 queue with arrival rate $\lambda_i = \lambda$ and service rate $\mu_i = \mu$. Now Let $q_i^+ = q^+$ and $q_i^- = q^-$ for all $i$ where $q^-$ and $q^+$ are two constants in $(0,1]$.

In this case, we have $P_k - \Lambda_k = (\pi_k \, q^-)/(q^+ + q^-)$
which leads to
\[
\RRR_\infty = 1 + H(q^-, q^+), 
\]
where $H(q^-, q^+)$ is the harmonic mean of $q^-$ and $q^+$. In particular if $q^- = q^+ = q$ we get $\RRR_\infty = 1 + q$.

Interestingly, when either $q^+ = 0$ or $q^- = 0$, the counting process becomes a thinned Poisson process  and we obtain $\RRR_\infty = \RRR = 1$. However, when both the input and output processes are counted together, we have $\RRR_\infty > 1$ due to the dependence and positive correlation between the two thinned processes.
\end{example}
\subsection*{Acknowledgements}%
Daryl Daley's work done as an Honorary Professorial Associate at the University of Melbourne. The second and third author are indebted to his mentorship and collegiality throughout the work on this project. Daryl Daley, passed on 16 April 2023 and the last face to face discussion about this manuscript was a month earlier. See \cite{Taylor} for an insightful outline of some his scientific contributions throughout his career. 

Jiesen Wang's research has been supported by the European Union’s Horizon 2020 research and innovation programme under the Marie Sklodowska-Curie grant agreement no. 101034253, and by the NWO Gravitation project NETWORKS under grant agreement no. 024.002.003 \includegraphics[height=1em]{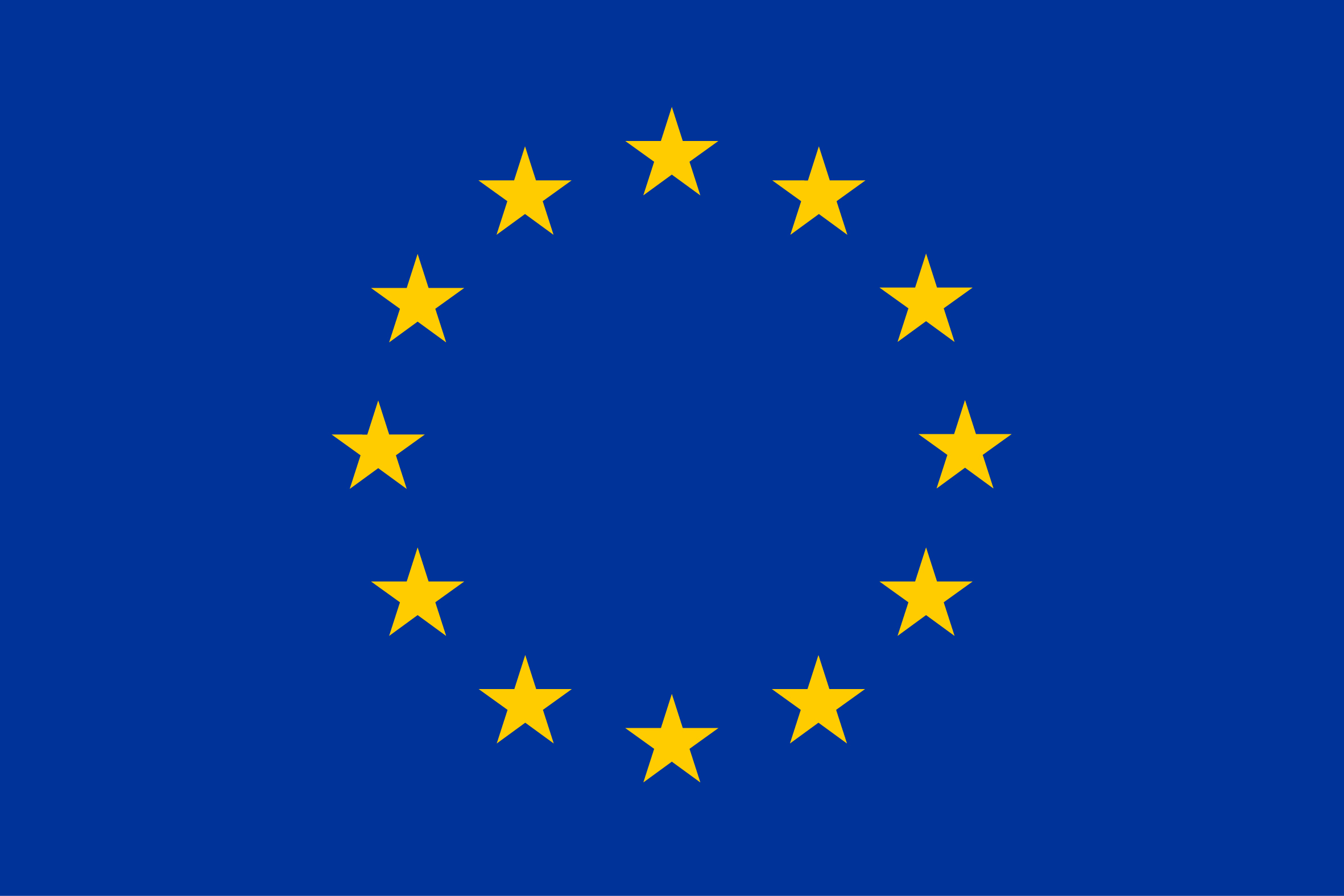}.



\begin{thebibliography}{99}  

\bibitem{al2011asymptotic}
{\sc Al Hanbali, A., Mandjes, M., Nazarathy, Y., and Whitt, W.} (2011).
The asymptotic variance of departures in critically loaded queues.
{\it Advances in Applied Probability \bf43}, 243--263.

\bibitem{asmus2003}
{\sc Asmussen, S.} (2003).
{\sl Applied Probability and Queues}, second edition.  
Springer, New York.


\bibitem{brown1974soa}
{\sc Brown, M. and Solomon, H.} (1975).
A second-order approximation for the variance of a renewal reward process.
{\it Stochastic Processes and their Applications \bf3}, 301--314.



\bibitem{daley2011revisiting}
{\sc Daley, D.J.} (2011).
Revisiting queueing output processes: a point process viewpoint.
{\it Queueing Systems \bf68}, 395--405.

\bibitem{daley2013bravo}
{\sc Daley, D.J., Nazarathy, Y., and van Leeuwaarden, J.S.H.} (2015).
BRAVO for many-server QED systems with finite buffers.
{\it Advances in Applied Probability \bf47}, 231--250.


\bibitem{EdenK}
{\sc Eden, U.T. and Kramer, M.A.} (2010).
Drawing inferences from Fano factor calculations.
{\it Journal of Neuroscience Methods \bf190}, 231--250.


\bibitem{glynn2023heavy}
{\sc Glynn, P.W. and Wang, R.J.} (2023).
A heavy-traffic perspective on departure process variability.
{\it Stochastic Processes and their Applications \bf166}, 104097


\bibitem{hautphenne2013second}
{\sc Hautphenne, S., Kerner, S., Nazarathy, Y., and Taylor, P.G.} (2015).
The intercept term of the asymptotic variance curve for some queueing output processes.
{\it European Journal of Operational Research \bf24}, 455--464.

\bibitem{he2014fundamentals}
{\sc He, Q. M.} (2014).
{\sl Fundamentals of Matrix-Analytic Methods}.
Springer, Vol. 365.

\bibitem{kellyBook}
{\sc Kelly, F.P.} (2011).
{\sl Reversibility and Stochastic Networks}.
Cambridge University Press, Cambridge.

\bibitem{nadarajah2018precise}
{\sc Nadarajah, S. and Pogány, T.K.} (2018).
Precise formulae for Bravo coefficients.
{\it Operations Research Letters \bf46}, 189--192.

\bibitem{nazarathy2011variance}
{\sc Nazarathy, Y.} (2011).
The variance of departure processes: puzzling behavior and open problems.
{\it Queueing Systems \bf68}, 385--394.

\bibitem{nazarathy2022busy}
{\sc Nazarathy, Y. and Palmowski, Z.} (2022).
On busy periods of the critical GI/G/1 queue and BRAVO.
{\it Queueing Systems \bf102}, 219--225.

\bibitem{NazarathyWeiss0336}
{\sc Nazarathy, Y. and Weiss, G.} (2008).
The asymptotic variance rate of finite capacity birth--death queues.
{\it Queueing Systems \bf59}, 135--156.




\bibitem{Taylor}
{\sc Taylor, P.} (2023).
Daryl John Daley, 4 April 1939--16 April 2023: An internationally acclaimed researcher in applied probability and a gentleman of great kindness.
{\it Journal of Applied Probability \bf60}, 1516--1531.

{\color{black}
\bibitem{weiss2021scheduling}
Weiss, G.
\newblock \emph{Scheduling and control of queueing networks}.
\newblock Cambridge University Press, volume 14, 2021.
}
\end{thebibliography}




\bigskip 

\end{document}